\documentclass[4paper,11pt]{amsart}
\usepackage[utf8]{inputenc}
\usepackage[english]{babel}

\usepackage{biblatex}
\usepackage{graphicx} 
 
\usepackage{xcolor}

\usepackage{breqn} 

\usepackage{amsfonts}
\usepackage{amsthm}
\usepackage{amsmath}
\usepackage{mathtools}

\usepackage{textcomp}
\usepackage{multirow}
\usepackage{xfrac}
\usepackage{array, tabularx}
\usepackage{setspace}

\newtheorem*{theorem*}{Theorem}
\newtheorem*{prop*}{Proposition}
\newtheorem{prop}{Proposition}
\newtheorem{theorem}{Theorem}
\newtheorem*{definition*}{Definition}
\newtheorem*{fact*}{Fact}
\newtheorem{lemma}{Lemma}
\newtheorem*{lemma*}{Lemma}
\newtheorem*{rmk*}{Remark}
\newtheorem{rmk}{Remark}
\newtheorem*{corol*}{Corollary}
\newtheorem{corol}{Corollary}

\title[Bismut-Griffiths positivity and HCF]{Griffiths positivity for Bismut curvature and its behaviour along Hermitian Curvature Flows}
\author{Giuseppe Barbaro}
\address{Dipartimento di Matematica ``Guido Castelnuovo", Università la Sapienza, Piazzale Aldo Moro, 5, 00185 Roma, Italy} 
\email{g.barbaro@uniroma1.it}

\begin{document}

\maketitle

\begin{abstract}
    In this note we study a positivity notion for the curvature of the Bismut connection; more precisely, we study the notion of \emph{Bismut-Griffiths-positivity} for complex Hermitian non-K{\"a}hler manifolds. Since the K{\"a}hler-Ricci flow preserves and regularizes the usual Griffiths positivity we investigate the behaviour of the Bismut-Griffiths-positivity under the action of the Hermitian curvature flows. In particular we study two concrete classes of examples, namely, linear Hopf manifolds and six-dimensional Calabi-Yau solvmanifolds with holomorphically-trivial canonical bundle. From these examples we identify some HCFs which do not preserve Bismut-Griffiths-non-negativity.
\end{abstract}

\section{Introduction}
In this note we introduce a notion of positivity (or non-negativity) for complex Hermitian manifolds which emulates the definition of Griffiths positivity (non-negativity) and refers to the Bismut curvature tensor. Then we investigate its behaviour under the action of the Hermitian curvature flows. 

Let us now introduce our problem giving more details.
The Bismut connection $\nabla^{+}$ on a Hermitian manifold $(M,g,J)$ is the unique Hermitian connection with totally skew-symmetric torsion. It can be defined by the formula 
\begin{equation*}
    g(\nabla^{+}_{X}Y,Z)=g(\nabla^{LC}_{X}Y,Z)+\frac{1}{2}Jd\omega(X,Y,Z)
\end{equation*}
where $\nabla^{LC}$ is the Levi-Civita connection, $\omega$ is the canonical $2$-form associated to $g$ and $J$ acts as $Jd\omega(\cdot,\cdot,\cdot)=-dw(J\cdot,J\cdot,J\cdot)$.

We define a notion of positivity emulating the definition of the Griffiths positivity for the Chern connection of the holomorphic tangent bundle (a.k.a. holomorphic bisectional curvature). We do this by evaluating the holomorphic Bismut bisectional curvature of the Hermitian manifold.
\begin{definition*}
    A Hermitian manifold $(M,g,J)$ has \emph{Bismut-Griffiths-positive} (resp. \emph{non-negative}) curvature if its Bismut curvature tensor $\Omega^{B}$  satisfies
    \begin{itemize}
        \item  $\Omega^B\in\wedge^{1,1}M\otimes\wedge^{1,1}M$;
        \item for any non-zero $\xi,\nu\in T^{1,0}M$  
        \begin{equation*}
            \Omega^{B}(\xi,\overline{\xi},\nu,\overline{\nu})>0 \hspace{1em} (resp. \geq 0)\;.
        \end{equation*}
    \end{itemize} 
\end{definition*}

The condition $\Omega\in\wedge^{1,1}M\otimes\wedge^{1,1}M$ is known in the literature as (Cplx). We note that the curvature tensor $\Omega$ associated to a Hermitian connection satisfies (Cplx) if and only if it satisfies the $J$-invariance formula: $\Omega_{ijk\overline{l}}=0$.\\
We could define our notion of positivity even if (Cplx) were not satisfied; however, in that case we would only describe the geometry of the $(1,1)$ part of $\Omega^{B}$ ignoring the $(2,0)$ and $(0,2)$ components.
\begin{rmk*}
    Given a complex Hermitian manifold $(M,g,J)$, if the metric is pluriclosed, the Bismut curvature tensor (in special coordinate around $z$) become:
    \begin{equation*}
        \Omega^{B}_{i\overline{j}k\overline{l}}=\Omega^{Ch}_{k\overline{l}i\overline{j}} - g^{p\overline{q}}T_{kp\overline{j}}\overline{T_{lq\overline{i}}}
    \end{equation*}
    Thus for SKT manifolds the Bismut-Griffiths positivity (non-negativity) implies the Griffiths positivity (non-negativity).\\
    In \cite{tong} F. Tong studies a positivity notion for the tensor $\Omega^{Ch}_{k\overline{l}i\overline{j}} - g^{p\overline{q}}T_{kp\overline{j}}\overline{T_{lq\overline{i}}}$ which naturally arises from a Bochner-type formula for closed $(1,1)$-forms.
\end{rmk*}

We are interested in the behaviour of the Bismut-Griffiths positivity under the action of the \textit{Hermitian Curvature Flows} (HCFs). In the article \cite{Streets_2011}, Streets and Tian suggest this class of flows as a new class of parabolic flows of metrics on Hermitian manifolds, and proved short time existence and regularity results. These flows have been used to reveal information about the geometry of the varieties, see for example the works of Streets and Tian on the \textit{Pluriclosed flow} \cite{streets2018pluriclosed}, \cite{Streets_2010} and Ustinovskiy \cite{MR3828497}, \cite{Ustinovskiy_2019} and \cite{Ustinovskiy_2020}. In particular, we are motivated by the possibility of detecting some regularization properties as in \cite{Ustinovskiy_2019}. In that article Ustinovskiy showed that there is a flow in the HCF family (which we will call \emph{Ustinovskiy flow}) that not only preserves Griffiths positivity and non-negativity of the Chern connection, but it evolves a metric with non-negative Griffiths curvature everywhere and positive in some point to a metric with positive Griffiths curvature everywhere.\\ 
The HCFs are defined by the equation
\begin{equation*}
        \frac{\partial}{\partial t}g=-S+Q
\end{equation*}
where $S$ is the trace of the Chern curvature tensor $\Omega^{Ch}$
\begin{equation*}
    S_{i\overline{j}}=(Tr_{\omega}\Omega^{Ch})_{i\overline{j}}=g^{k\overline{l}}\Omega^{Ch}_{k\overline{l}i\overline{j}}
\end{equation*}
and $Q$ is a quadratic polynomial in the torsion $T^{Ch}$ of the Chern connection. More precisely, the components of the quadratic term $Q$ are:
\begin{align*}
	&Q^{1}_{i\overline{j}}=g^{k\overline{l}}g^{m\overline{n}}T^{Ch}_{ik\overline{n}}T^{Ch}_{\overline{jl}m}
    &&Q^{2}_{i\overline{j}}=g^{k\overline{l}}g^{m\overline{n}}T^{Ch}_{km\overline{j}}T^{Ch}_{\overline{ln}i} \\
    &Q^{3}_{i\overline{j}}=g^{k\overline{l}}g^{m\overline{n}}T^{Ch}_{ik\overline{l}}T^{Ch}_{\overline{jn}m}
    &&Q^{4}_{i\overline{j}}=\frac{1}{2}g^{k\overline{l}}g^{m\overline{n}}\left(T^{Ch}_{mk\overline{l}}T^{Ch}_{\overline{nj}i}+T^{Ch}_{mi\overline{j}}T^{Ch}_{\overline{nl}k}\right)
\end{align*}

Studying the notion of Bismut-Griffiths positivity, we focus on six dimensional Calabi-Yau solvmanifolds. These are compact quotients of solvable Lie groups endowed with invariant complex structures and with holomorphically trivial canonical bundle. In this class we find examples of manifolds which do not satisfy (Cplx) (see Theorem \ref{th_cplx solvmanifolds}) and among that which satisfy (Cplx) we find Bismut-Griffiths-non-negative manifolds. We also prove that
\begin{theorem*}[Theorem \ref{th_HCF on solvmanifolds}]
    Let $M$ be a six-dimensional Calabi-Yau solvmanifold, then any Hermitian curvature flow preserves (Cplx). Moreover, Hermitian curvature flows preserve Bismut-Griffiths-non-negativity on these manifolds.
\end{theorem*}
We also study our problem on linear Hopf manifolds. The linear Hopf surface with its standard metric ($g_{H}$, see \S \ref{sec_Hopf}) has flat Bismut curvature, thus it is our first example of Bismut-Griffiths-non-negative manifold. We also find other (non flat) examples of metrics with Bismut-Griffiths-non-negative curvature on linear Hopf manifolds of higher dimension. 

Our results come from the analysis of the class of $g(\alpha,\beta)$ metrics on linear Hopf manifolds (see $\S$ \ref{sec_Hopf}). These are all the homogeneous metrics on linear Hopf manifolds of dimension greater than two, while in dimension two they are all the $S^1 \times U(2)$-invariant metrics on the Hopf surface (see Proposition \ref{prop: homogeneous}). This ensures that they are closed by the action of any HCF and naturally arise performing the HCFs on the standard metric $g_{H}$. We prove
\begin{theorem*}[Proposition \ref{prop: Cplx satisfied on Hopf} \& Corollary \ref{cor: HCF preseves Cplx on Hopf}]
    On a linear Hopf manifold equipped with a $g(\alpha,\beta)$ metric the Bismut curvature tensor satisfies (Cplx). Moreover, on these manifolds any HCF starting from those metrics preserves (Cplx).
\end{theorem*}
We characterize the metrics $g(\alpha,\beta)$ which have Bismut-Griffiths-non-negative curvature (\S \ref{sec_canonical metrics on LHM}); then we give a description of the evolution of the HCFs in these family through a stability result (Theorem \ref{th_stability}), so we prove
\begin{theorem*}[Theorem \ref{Th inequality all dimensions} \& Proposition \ref{pr_sharpness}]
    There exists a class in the family HCF of flows which do not preserve Bismut-Griffiths-non-negativity; all the others HCFs preserve Bismut-Griffiths-non-negativity on linear Hopf manifolds equipped with  $g(\alpha,\beta)$ metrics.
\end{theorem*}
Finally, we check these conditions on some interesting Hermitian curvature flows, such as the Ustinovskiy flow and the pluriclosed and the \emph{gradient} flows of Streets and Tian (the latter is the only HCF which is a gradient flow for some functional $\mathbb{F}$, see \cite{Streets_2011}). In particular, the Ustinovskiy flow does not preserve the Bismut-Griffiths-non-negativity, while the gradient flow does on linear Hopf manifolds with $g(\alpha,\beta)$ metrics; the pluriclosed flow preserves the Bismut-Griffiths-non-negativity on the Hopf surface while it does not on linear Hopf manifolds of higher dimension. See \S \ref{sec_interesting flows} for details.

\section{Bismut-Griffiths-positivity of 6-dimensional Calabi-Yau solvmanifolds}\label{sec_solvmanifolds}
In this section we analyze the symmetries of (Cplx) and the notion of Bismut-Griffiths-positivity by investigating them on 6-dimensional Calabi-Yau solvmanifolds. By solvmanifold we mean a compact quotient of a connected simply-connected solvable Lie group by a co-compact discrete subgroup. We endow it with a Hermitian structure $(g,J)$ which is invariant (under left-translations) when lifted to the universal cover; moreover, we ask these solvmanifolds to be Calabi-Yau, that is, the complex structure $J$ is such that the canonical bundle is holomorphically-trivial. This includes nilmanifolds with invariant complex structures.\\
We refer to the classification (up to linear equivalence) of the invariant complex structures on six-dimensional nilmanifolds (Table 1) and solvmanifolds non-nilmanifolds with holomorphically-trivial canonical bundle (Table 2) as outlined in the works of Salamon, Ugarte, Villacampa, Andrada, Barberis, Dotti, Ceballos and Otal  \cite{Salamon_2001}, \cite{Ugarte_2007}, \cite{Andrada_2011}, \cite{Ugarte_2014}, \cite{article}. 

\begin{table}[h!]
\renewcommand*{\arraystretch}{1.6}
\begin{center}
{\resizebox{\textwidth}{!}{
\begin{tabular}{|c|c|l|}
\hline
\textbf{Name}&\textbf{Complex structure}&\textbf{Lie algebra}\\ \hline\hline
\multirow{2}{*}{\text{(Np)}}&\multirow{2}{*}{$d\varphi^1=d\varphi^{2} = 0,\,\, d\varphi^3=\rho\, \varphi^{12},\,\, \text{ where }\rho\in\{0,1\}$}& $\rho=0: \mathfrak{h}_1=(0,0,0,0,0,0)$  \\ \cline{3-3}
&&$\rho=1: \mathfrak{h}_5=(0,0,0,0,13+42,14+23)$\\ \hline  \hline
\multirow{6}{*}{\text{(Ni)}}&\multirow{2}{*}{$d\varphi^1=d\varphi^{2} = 0$,}&$\mathfrak{h}_2=(0,0,0,0,12,34)$ \\ \cline{3-3}
&& $\mathfrak{h}_3=(0,0,0,0,0,12+34)$ \\ \cline{3-3}
&$d\varphi^3= \rho\, \varphi^{12} + \varphi^{1\bar1} + \lambda\,\varphi^{1\bar 2} + D\,\varphi^{2\bar2}$,& $\mathfrak{h}_4=(0,0,0,0,12,14+23)$ \\ \cline{3-3}
&& $\mathfrak{h}_5=(0,0,0,0,13+42,14+23)$ \\ \cline{3-3}
&where $\rho\in\{0,1\}, \lambda\in\mathbb R^{\geq0}, D\in\mathbb C \text{ with }\Im D\geq0 $& $\mathfrak{h}_6=(0,0,0,0,12,13)$ \\ \cline{3-3}
&& $\mathfrak{h}_8=(0,0,0,0,0,12)$ \\ \hline \hline

\multirow{9}{*}{\text{(Nii)}}&\multirow{3}{*}{$d\varphi^1=0, \quad d\varphi^{2}=\varphi^{1\bar1}$,}&$\mathfrak{h}_7=(0,0,0,12,13,23)$ \\ \cline{3-3}
&& $\mathfrak{h}_9=(0,0,0,0,12,14+25)$ \\ \cline{3-3}
&& $\mathfrak{h}_{10}=(0,0,0,12,13,14)$ \\ \cline{3-3}
&\multirow{2}{*}{$d\varphi^3=\rho\varphi^{12}+B\,\varphi^{1\bar2}+c\,\varphi^{2\bar1}$,}& $\mathfrak{h}_{11}=(0,0,0,12,13,14+23)$ \\ \cline{3-3}
&& $\mathfrak{h}_{12}=(0,0,0,12,13,24)$ \\ \cline{3-3}
&\multirow{3}{*}{where $\rho\in\{0,1\}, B \in \mathbb C, c\in\mathbb R^{\geq0} , \text{ with } (\rho,B,c)\neq(0,0,0)$}& $\mathfrak{h}_{13}=(0,0,0,12,13+14,24)$ \\ \cline{3-3}
&& $\mathfrak{h}_{14}=(0,0,0,12,14,13+42)$ \\ \cline{3-3}
&& $\mathfrak{h}_{15}=(0,0,0,12,13+42,14+23)$\\ \cline{3-3}
&& $\mathfrak{h}_{16}=(0,0,0,12,14,24)$ \\ \hline \hline
\multirow{2}{*}{\text{(Niii)}}&$d\varphi^1=0,\quad d\varphi^{2} =\varphi^{13} + \varphi^{1\bar 3}$,&$\mathfrak{h}_{19}^-=(0,0,0,12,23,14-35)$ \\ \cline{3-3}
&$d\varphi^3=\sqrt{-1}\rho\, \varphi^{1\bar1}\pm \sqrt{-1}(\varphi^{1\bar2} - \varphi^{2\bar1}),\quad \text{ where }\rho\in\{0,1\}$&$\mathfrak{h}_{26}^+=(0,0,12, 13,23,14+25)$ \\ \hline  \hline
 \end{tabular}
 }}
 \caption{Invariant complex structures on six-dimensional nilmanifolds up to linear equivalence, see \cite{Andrada_2011}, \cite{article}, \cite{Ugarte_2014}.}
\end{center}
\end{table}

\begin{table}[h!]
\renewcommand*{\arraystretch}{1.6}
\begin{center}
{\resizebox{\textwidth}{!}{
\begin{tabular}{|c|c|l|}
\hline
\textbf{Name}&\textbf{Complex structure}&\textbf{Lie algebra}\\ \hline\hline
\multirow{3}{*}{\text{(Si)}}&$d\varphi^1=A\varphi^{13}+A\varphi^{1\bar3}$,&$\mathfrak{g}_1=(15,-25,-35,45,0,0)$  when $\theta=0$  \\ \cline{3-3} 
&$d\varphi^{2}=-A\varphi^{23}-A\varphi^{2\bar3},\quad d\varphi^3=0$,& $\mathfrak{g}_2^{\alpha}=(\alpha \times 15+25, -15+\alpha\times 25, -\alpha\times 35+45, -35-\alpha\times 45,0,0)$\\
&where $A = \cos\theta+\sqrt{-1}\sin\theta, \theta\in[0,\pi)$&with $\alpha=\frac{\cos\theta}{\sin\theta}\geq0$, when $\theta\neq0$\\ \hline  \hline
\multirow{3}{*}{\text{(Sii)}}&$d\varphi^1=0,\quad d\varphi^{2}=-\frac{1}{2}\varphi^{13}-\left(\frac{1}{2}+\sqrt{-1} x\right) \varphi^{1\bar3}+\sqrt{-1}x\,\varphi^{3\bar1},$&\multirow{3}{*}{$\mathfrak{g}_3=(0, -13, 12, 0, -46, -45)$} \\
&$d\varphi^3=\frac{1}{2}\varphi^{12}+\left(\frac{1}{2}-\frac{\sqrt{-1}}{4x}\right)\varphi^{1\bar2}+\frac{\sqrt{-1}}{4x}\varphi^{2\bar1}$,&  \\
&where $x\in\mathbb R^{>0}$&\\ \hline \hline
\multirow{3}{*}{\text{(Siii1)}}&$d\varphi^1=\sqrt{-1}\varphi^{13}+\sqrt{-1}\varphi^{1\bar3}$&\multirow{3}{*}{$\mathfrak{g}_4=(23, -36, 26, -56, 46, 0)$} \\
&$d\varphi^{2}=-\sqrt{-1}\varphi^{23}-\sqrt{-1}\varphi^{2\bar3}$&  \\
&$d\varphi^3=\pm \varphi^{1\bar1}$&\\ \hline
\multirow{3}{*}{\text{(Siii2)}}&$d\varphi^1=\varphi^{13}+\varphi^{1\bar3}$&\multirow{3}{*}{$\mathfrak{g}_5=(24 + 35, 26, 36, -46, -56, 0)$} \\
&$d\varphi^{2}=-\varphi^{23}-\varphi^{2\bar3}$&  \\
&$d\varphi^3=\varphi^{1\bar2}+\varphi^{2\bar1}$&\\ \hline
\multirow{3}{*}{\text{(Siii3)}}&$d\varphi^1=\sqrt{-1}\varphi^{13}+\sqrt{-1}\varphi^{1\bar3}$&\multirow{3}{*}{$\mathfrak{g}_6=(24 + 35, -36, 26, -	56, 46, 0)$} \\
&$d\varphi^{2}=-\sqrt{-1}\varphi^{23}-\sqrt{-1}\varphi^{2\bar3}$&  \\
&$d\varphi^3=\varphi^{1\bar1}+\varphi^{2\bar2}$&\\ \hline
\multirow{3}{*}{\text{(Siii4)}}&$d\varphi^1=\sqrt{-1}\varphi^{13}+\sqrt{-1}\varphi^{1\bar3}$&\multirow{3}{*}{$\mathfrak{g}_7=(24 + 35, 46, 56, -26, -36, 0)$} \\
&$d\varphi^{2}=-\sqrt{-1}\varphi^{23}-\sqrt{-1}\varphi^{2\bar3}$&  \\
&$d\varphi^3=\pm(\varphi^{1\bar1}-\varphi^{2\bar2})$&\\ \hline  \hline
\text{(Siv1)}&$d\varphi^1=-\varphi^{13}, \quad d\varphi^{2}=\varphi^{23},\quad d\varphi^3=0$&\multirow{6}{*}{$\mathfrak{g}_8=(16-25, 15+26, -36+45, -35-46, 0, 0)$} \\ \cline{1-2}
\multirow{2}{*}{\text{(Siv2)}}&$d\varphi^1=2\sqrt{-1}\varphi^{13}+\varphi^{3\bar3},\quad x \in \{0,1\}$ & \\ &$d\varphi^{2}=-2\sqrt{-1}\varphi^{23}+x\,\varphi^{3\bar3}, \quad d\varphi^3=0$&\\ \cline{1-2}
\multirow{3}{*}{\text{(Siv3)}}&$d\varphi^1=A\,\varphi^{13}-\varphi^{1\bar3}$ & \\ &$d\varphi^{2}=-A\,\varphi^{23}+\varphi^{2\bar3}, \quad d\varphi^3=0$&\\
&$ A \in \mathbb C \text{ with } |A| \neq 1$&\\ \hline \hline
\multirow{3}{*}{\text{(Sv)}}&$d\varphi^1=-\varphi^{3\bar3}$&\multirow{3}{*}{$\mathfrak{g}_9=(45, 15 + 36, 14 - 26 + 56, -56, 46, 0)$} \\
&$d\varphi^{2}=\frac{\sqrt{-1}}{2}\varphi^{12}+\frac{1}{2}\varphi^{1\bar3}-\frac{\sqrt{-1}}{2}\varphi^{2\bar1}$&  \\
&$d\varphi^3=-\frac{\sqrt{-1}}{2}\varphi^{13}+\frac{\sqrt{-1}}{2}\varphi^{3\bar1}$&\\ \hline \hline
 \end{tabular}
 }}
 \caption{Invariant complex structures on six-dimensional solvmanifolds non-nilmanifolds with holomorphically-trivial canonical bundle up to linear equivalence, see \cite{ota14}, \cite{MR3436163}.}
 \label{table:solv-cplx}
\end{center}
\end{table}

In the formulas above the authors refer to a co-frame $(\varphi^{1},\varphi^{2},\varphi^{3},\overline{\varphi}^{1},\overline{\varphi}^{2},\overline{\varphi}^{3})$ where $(\varphi^{1},\varphi^{2},\varphi^{3})$ is an invariant co-frame of $(1,0)$-forms with respect to $J$.\\
The generic invariant Hermitian structure $\omega = g(J\cdot,\cdot)$ is given by
\begin{equation}\label{meetric on nilmanifolds}
    2 \omega = \sqrt{-1}(r^{2}\varphi^{1\overline{1}} + s^{2}\varphi^{2\overline{2}} + t^{2}\varphi^{3\overline{3}}) + 
    u\varphi^{1\overline{2}} - \overline{u}\varphi^{2\overline{1}} + 
    v\varphi^{2\overline{3}} - \overline{v}\varphi^{3\overline{2}} +
    z\varphi^{1\overline{3}} - \overline{z}\varphi^{3\overline{1}}
\end{equation}
where $\varphi^{i\overline{j}}=\varphi^{i}\land \overline{\varphi}^{j}$ and the coefficients satisfy the following inequalities coming from the fact that $g$ is positive definite (see \cite{Ugarte_2007}):
    $$r^{2}>0, \hspace{0.4em} s^{2}>0, \hspace{0.4em} t^{2}>0$$
    $$r^{2}s^{2}>|u|^{2}, \hspace{0.4em} r^{2}t^{2}>|z|^{2}, \hspace{0.4em} s^{2}t^{2}>|v|^{2}$$
    $$8\sqrt{-1}\det\Xi=r^{2}s^{2}t^{2}+2Re(\sqrt{-1}\overline{uv}z)-(r^{2}|v|^{2}+t^{2}|u|^{2}+s^{2}|z|^{2})>0$$
where, $\Xi$ denotes the Hermitian matrix associated to the Hermitian structure, i.e.
\begin{equation*}
    \Xi=
        \begin{pmatrix}
            \sqrt{-1}\frac{r^{2}}{2} & \frac{u}{2} & \frac{z}{2} \\
            -\frac{\overline{u}}{2} & \sqrt{-1}\frac{s^{2}}{2} & \frac{v}{2} \\
            -\frac{\overline{z}}{2} & -\frac{\overline{v}}{2} & \sqrt{-1}\frac{t^{2}}{2} 
        \end{pmatrix}
\end{equation*}

Analyzing case by case the possible families of nilmanifolds and solvmanifolds, we get the following results, whose proofs are collected in the Appendix.
\begin{theorem}\label{th_cplx solvmanifolds}
    Let $M$ be a six-dimensional solvmanifold endowed with invariant metric $g$ and complex structure $J$, $g$ as in (\ref{meetric on nilmanifolds}) and  $J$ such that the canonical bundle is holomorphically-trivial. The Bismut curvature tensor satisfies the (Cplx) condition precisely in the cases (Np), (Ni), (Nii), (Si), (Siii1), (Siv1) and (Siv3) when the conditions on the invariant structures of Table \ref{table:Sum up} are satisfied.
\end{theorem}

\begin{table}[h!]
\renewcommand*{\arraystretch}{1.6}
\begin{center}
{\resizebox{0.80\textwidth}{!}{
\begin{tabular}{|c|c|l|}
\hline
\textbf{Name}&\textbf{(Cplx) condition}&\textbf{Bismut-Griffiths-non-negativity}\\ \hline\hline
\multirow{2}{*}{\text{(Np)}}& Always satisfied & $\rho=0:$ flat\\
& & $\rho=1: $ nowhere non-negative nor non-positive\\ \hline  
\multirow{6}{*}{\text{(Ni)}}& &$\mathfrak{h}_{2}$: non-negative if $u=0$\\
& & $\mathfrak{h}_{3}$, $D=1$: non-negative \\
& & $\mathfrak{h}_{3}$, $D=-1$: nowhere non-negative nor non-positive \\
& $\rho=0$ & $\mathfrak{h}_{4}$: nowhere non-negative nor non-positive\\
& & $\mathfrak{h}_{5}$: nowhere non-negative nor non-positive\\
& & $\mathfrak{h}_{8}$: non-negative\\
\hline
\multirow{1}{*}{\text{(Nii)}}&$c=B=0$, $\rho=1$, $v=0$ & nowhere non-negative nor non-positive\\ 
\hline
\hline 
\multirow{2}{*}{\text{(Si)}}& $u=v=z=0$ & $A=\sqrt{-1}$: flat\\
& & $A\neq \sqrt{-1}$: nowhere non-negative nor non-positive\\
\hline
\multirow{1}{*}{\text{(Siii1)}}& $u=v=z=0$ &  non-negative \\
\hline
\multirow{1}{*}{\text{(Siv1)}}& Always satisfied & nowhere non-negative nor non-positive  \\ 
\hline
\multirow{2}{*}{\text{(Siv3)}}& $u=v=z=0$ &  nowhere non-negative nor non-positive \\& $A=0$, $v=z=0$ & in both cases\\
\hline
\hline 

 \end{tabular}
 }}
 \caption{Conditions on the underlying complex structure, invariant Hermitian metric and Lie algebra.}
 \label{table:Sum up}
\end{center}
\end{table}

\begin{rmk*}
    In \cite{angella2018gauduchon} the authors studied the existence of Gauduchon K{\"a}hler-like connections on 6-dimensional Calabi-Yau solvmanifolds. The Gauduchon connections are an affine line of Hermitian connections which goes through the Chern and the Bismut connections. Moreover, K{\"a}hler-like means that the curvature tensor satisfies both (Cplx) and the first Bianchi identity. Examples on the Hopf manifolds show that the K{\"a}hler-like condition is strictly stronger than (Cplx).
\end{rmk*}

In the cases where (Cplx) is satisfied we look at the holomorphic bisectional Bismut curvature.
\begin{theorem}\label{prop_positivity solvmanifolds}
    Let $M$ be a six-dimensional solvmanifold endowed with invariant metric $g$ and complex structure $J$, $g$ as in (\ref{meetric on nilmanifolds}) and  $J$ such that the canonical bundle is holomorphically-trivial. If $J$ is in the families (Siii1) or (Ni) with Lie algebra $\mathfrak{h}_{2}$, $\mathfrak{h}_{8}$ and $\mathfrak{h}_{3}$ (with $D=1$)  then the Bismut curvature tensor is Bismut-Griffiths non-negative. If $J$ is in the family (Si) with Lie algebra $\mathfrak{g}_{2}^{0}$ and diagonal metric the manifold is Bismut-flat. In all the other cases where (Cplx) is satisfied the invariant metrics are neither non-positive nor non-negative. (See Table \ref{table:Sum up})
\end{theorem}

The computations in the Appendix lead to the Remarks \ref{rmk: 1}, \ref{rmk: 2}, \ref{rmk: 3}, \ref{rmk: 4} and \ref{rmk: 5} that we summarize in the following statement.
\begin{theorem}\label{th_HCF on solvmanifolds}
    Let $M$ be a six-dimensional solvmanifold endowed with invariant metric $g$ and complex structure $J$, $g$ as in (\ref{meetric on nilmanifolds}) and  $J$ such that the canonical bundle is holomorphically-trivial. Then the symmetries of (Cplx) are preserved by any HCF. Moreover, the HCFs preserve Bismut-Griffiths-non-negativity and Bismut-flatness, when they occur.
\end{theorem}
We remark that recent results in \cite{fino2021pluriclosed} show that the \emph{pluriclosed flow} (which is in the HCF family) preserves the Bismut K{\"a}hler-like condition on $2$-step nilpotent Lie group with left-invariant Hermitian structure.

\medskip

All the computations on six-dimensional Calabi-Yau solvmanifolds within the proofs of the above statements are contained in the Appendix.

\section{HCFs on Linear Hopf Manifolds}\label{sec_Hopf}
Linear Hopf manifolds are defined as quotients of the complex domain $\mathbb{C}^{n}\setminus\{0\}$ over an equivalence relation depending on $a\in\mathbb{C}^{n}\setminus\{0\}$ 
\begin{equation*}
    M^{n}=\frac{\mathbb{C}^{n}\setminus \{0\}}{\sim}\;,
\end{equation*}
where $(z_{1},\cdots,z_{n})\sim(a_{1}z_{1},\cdots,a_{n}z_{n})$ with $|a_{1}|=\cdots=|a_{n}|\neq1$.\\
These manifolds come with a natural complex structure and a Hermitian metric 
\begin{equation*}
	g_{H}=\frac{\delta_{ij}}{|z|^{2}} \, dz^{i}\otimes d\overline{z}^{j}\;.
\end{equation*}
The Bismut curvature tensor associated to $g_{H}$ satisfies various symmetries, including (Cplx). Indeed, its non-vanishing coefficients are (see \cite{liu2012geometry}):
\begin{equation*}
	\Omega^{B}_{i\overline{j}k\overline{l}}(z)=\frac{\delta_{il}\delta_{jk}-\delta_{ij}\delta_{kl}}{|z|^{4}}+\frac{\delta_{ij}\overline{z}_{k}z_{l}+\delta_{kl}\overline{z}_{i}z_{j}-\delta_{il}z_{j}\overline{z}_{k}-\delta_{jk}\overline{z}_{i}z_{l}}{|z|^{6}}\;.
\end{equation*}
Thus, for any $\xi,\nu\in T^{1,0}M$
\begin{multline*}
        \Omega^{B}(\xi,\overline{\xi},\nu,\overline{\nu})_{|_{z}}=\frac{1}{|z|^{6}}\left(-|\xi|^{2}|\nu|^{2}|z|^{2}+|\xi\cdot\nu|^{2}|z|^{2}+|\nu\cdot z|^{2}|\xi|^{2}+|z\cdot\xi|^{2}|\nu|^{2}+\right.\\
	     \left.  -(\xi\cdot\nu)(\nu\cdot z)(z\cdot\xi)-(\nu\cdot\xi)(\xi\cdot z)(z\cdot\nu)\right)\;.
\end{multline*}
Since this vanishes for $n=2$, we get the following
\begin{prop*}
    For $n=2$ the Hopf manifold with canonical metric $g_{H}$ is Bismut flat. Thus, in particular, it is \textit{Bismut-Griffiths-non-negative}.
\end{prop*}
However, this is not the case in higher dimension. 
\begin{prop*}
    The Hopf manifold with canonical metric $g_{H}$ is not Bismut-Griffiths non-negative for $n>2$.
\end{prop*}
\begin{proof}
    The Bismut curvature tensor satisfies $\Omega^{B}_{i\overline{j}k\overline{l}}=-\Omega^{B}_{k\overline{j}i\overline{l}}=-\Omega^{B}_{i\overline{l}k\overline{j}}$. Thus $\Omega^{B}_{i\overline{j}k\overline{l}}=\Omega^{B}_{k\overline{l}i\overline{j}}$ and the two Bismut-Ricci curvatures agree and are
    \begin{equation*}
	    R^{B}_{i\overline{j}}=\frac{1}{|z|^{4}}\left((2-n)(\delta_{ij}|z|^{2}-\overline{z}_{i}z_{j})\right)\;,
    \end{equation*}
    so that for any $\xi\in T^{1,0}M$
    \begin{equation*}
	    R^{B}(\xi,\overline{\xi})=\frac{2-n}{|z|^{4}}\left(|\xi|^{2}|z|^{2}-|\xi\cdot z|^{2}\right)\leq 0\;,
    \end{equation*}
    and the equality holds only if $\xi=\lambda z$ with $\lambda\in\mathbb{C}$.\\
    Thus, if we take $\xi$ different from any multiples of $z$
    \begin{align*}
	    0 > R_z^{B}\left(\xi,\overline{\xi}\right)&=g^{i\overline{j}}\Omega_z^{B}\left(\xi,\overline{\xi},\partial_{i},\overline{\partial}_{j}\right)=|z|^{2}\sum_{i}\Omega_z^{B}\left(\xi,\overline{\xi},\partial_{i},\overline{\partial}_{i}\right)\;,
    \end{align*}
    so at least one of the $\Omega_z^{B}\left(\xi,\overline{\xi},\partial_{i},\overline{\partial}_{i}\right)$ is strictly negative.
\end{proof}

\subsection{HCF-closed family of metrics on linear Hopf manifolds}\label{sec_canonical metrics on LHM}
Here we study the Hermitian metrics $g(\alpha,\beta)$ (depending on real parameters $\alpha>0$ and $\beta>-\alpha$) on a generic linear Hopf manifold defined by
\begin{equation*}
	g(\alpha,\beta)_{i\overline{j}}=\alpha\frac{\delta_{ij}}{|z|^{2}}+\beta \frac{\overline{z}_{i}z_{j}}{|z|^{4}}\;.
\end{equation*}
\begin{prop}\label{prop: homogeneous}
    Given an $n$-dimensional linear Hopf manifold $M$, the $g(\alpha,\beta)$ metrics are all the $S^1 \times U(n)$-invariant metrics on $M$; hence, in particular, they are homogeneous. Moreover, if $n\geq 3$ they are all the Hermitian homogeneous metrics on $M$.
\end{prop}
\begin{proof}
    On a generic Hermitian metric on $M$
    $$ g=g_{i\overline{j}}(z)\, dz^i \otimes d\overline{z}^j \;,$$ the $U(n)$-invariant condition is
    $$ \overline{U} (g_{i\overline{j}}(p))U^t = (g_{i\overline{j}}(Up)) \;, $$ for any $U\in U (n) $. Notice that it is satisfied by the metrics $g_z(\alpha,\beta)$.
    Moreover, these metrics also are $S^1$ invariant, hence homogeneous on $M$.\\
    Now suppose that $g$ is an Hermitian $S^1 \times U(n)$-invariant metric on $M$. If we fix a point $p \in M$ and consider its isotropy group $(S^1 \times U(n))_p \cong U(n-1)$, we get that  
     $$ \overline{U} (g_{i\overline{j}}(p))U^t = (g_{i\overline{j}}(p)) \;, $$ for any $U\in (S^1 \times U(n))_p $.
    Moreover, we can take $e_1$ as point $p$; hence we get that the matrix
    $$\begin{pmatrix}
        1 &\\
          & U
    \end{pmatrix}
    \begin{pmatrix}
        g_{i\overline{j}}(e_1)
    \end{pmatrix}
    \begin{pmatrix}
        1 &   \\
          & U^\dagger
    \end{pmatrix}
    $$
    must be independent on $U\in U(n-1)$. \\
    The above equation forces $g_{i\overline{j}}(e_1)$ to be of the form 
    $$ g_{i\overline{j}}(e_1)=
    \begin{pmatrix}
        a &  \\
         & \lambda Id
    \end{pmatrix} \;,$$ where $a$ and $\lambda$ are positive real numbers. This means that $g$ agrees with $g(\lambda,a-\lambda)$ in $e_1$; moreover the $S^1\times U(n)-$invariance ensures that they agree all over $M$.\\
    The last statement comes from the fact that if $n>2$ on the sphere $S^{2n-1}$ any $SU(n)$-invariant metric actually is $U(n)$-invariant.
\end{proof}
Since the HCFs preserve the $S^1 \times U(n)$-invariance of the metrics, the $g(\alpha,\beta)$ family must be closed by the action of the HCFs. Moreover, the standard metric $g_{H}$ is $g(1,0)$, hence this family naturally arises studying the evolution of HCFs on linear Hopf manifolds. For the same reason, these metrics also arise in the evolution of $g_{H}$ by the Chern-Ricci flow (see \cite{tos-wei}). We also recall that the $g(\alpha,\beta)$ metrics also appear in \cite{liu-yang} where the authors use them to produce examples of Levi-Civita Ricci-flat Hermitian metrics on Hopf manifolds. \\
The inverse of $g(\alpha,\beta)$ is 
\begin{equation*}
	g(\alpha,\beta)^{i\overline{j}}=\frac{|z|^{2}}{\alpha}\left(\delta^{ij}-\frac{\beta}{\alpha+\beta}\frac{\overline{z}^{j}z^{i}}{|z|^{2}}\right)\;.
\end{equation*}
The Christoffel symbols for the Bismut connection are
\begin{align*}
	\Gamma_{ij}^{k}&=g^{k\overline{s}}\partial_{j}g_{i\overline{s}}=\frac{1}{|z|^{2}}\left(\frac{\beta}{\alpha}\delta_{j}^{k}\overline{z}_{i}-\delta_{i}^{k}\overline{z}_{j}\right)-\frac{\beta}{\alpha}\frac{\overline{z}_{i}\overline{z}_{j}z^{k}}{|z|^{4}} \;;\\
    \Gamma_{\overline{i}j}^{k} &=g^{k\overline{s}}\left(\overline{\partial}_{i}g_{j\overline{s}}-\overline{\partial}_{s}g_{j\overline{i}}\right)=\frac{1}{|z|^{2}}\left(\delta_{ij}z^{k}-\frac{\alpha+\beta}{\alpha}\delta_{j}^{k}z_{i}\right)+\frac{\beta}{\alpha}\frac{\overline{z}_{i}\overline{z}_{j}z^{k}}{|z|^{4}}\;,
\end{align*}
and a direct computation leads to $\Omega_{ijk\overline{l}}^{B}=0$ for any $i,j,k,l\in\{1,\ldots,n\}$. Thus we have the following,
\begin{prop}\label{prop: Cplx satisfied on Hopf}
	The Bismut curvature tensor associated to a $g(\alpha,\beta)$ metric on a linear Hopf manifolds satisfies (Cplx).
\end{prop}   
\begin{corol}\label{cor: HCF preseves Cplx on Hopf}
    The symmetries of (Cplx) are preserved by any HCF acting on a linear Hopf manifold equipped with a metric $g(\alpha,\beta)$.
\end{corol}
Now we turn to the study of the holomorphic bisectional Bismut curvature of these metrics.\\
We have the following formula for the Bismut curvature tensor of $g(\alpha,\beta)$:
\begin{align*}
	\Omega^{B}_{i\overline{j}k\overline{l}}=&\alpha\underbrace{\left[\frac{\delta_{il}\delta_{jk}-\delta_{ij}\delta_{kl}}{|z|^{4}}+\frac{\delta_{ij}\overline{z}_{k}z_{l}+\delta_{kl}\overline{z}_{i}z_{j}-\delta_{il}z_{j}\overline{z}_{k}-\delta_{jk}\overline{z}_{i}z_{l}}{|z|^{6}}\right]}_{U_\alpha}\\
    &+2\beta\underbrace{\left[\frac{-\delta_{ij}\delta_{kl}}{|z|^{4}}+\frac{\delta_{ij}\overline{z}_{k}z_{l}+\delta_{kl}\overline{z}_{i}z_{j}}{|z|^{6}}+\frac{-\overline{z}_{i}z_{j}\overline{z}_{k}z_{l}}{|z|^{8}}\right]}_{U_\beta} \; .
\end{align*}
Notice that we already know $U_{\alpha}$ since it is the curvature tensor of $g_{H}$ (w.r.t. the Bismut connection). If we evaluate $U_{\beta}$ on vectors $\xi,\nu\in T^{1,0}M$, we get
\begin{align}\label{eq_U_beta}
	U_{\beta}\left(\xi,\overline{\xi},\nu,\overline{\nu}\right)&=\frac{1}{|z|^{8}}\left(|\xi|^{2}|z|^{2}-|\xi\cdot z|^{2}\right)\left(|\nu\cdot z|^{2}-|\nu|^{2}|z|^{2}\right) \leq 0\;,
\end{align}
and the equality holds if and only if $\xi=\lambda z$ or $\nu=\lambda z$ with $\lambda\in\mathbb{C}$.\\ 
In dimension two $U_{\alpha}$ vanishes, thus $g(\alpha,\beta)$ is Bismut-Griffiths-non-negative if and only if $\beta\leq 0$. In dimension greater than $2$, we have seen that $U_{\alpha}$ is not non-negative, as well as $U_{\beta}$; hence the non-negativity of the tensor will only depend on the ratio $\gamma\doteq\frac{\beta}{\alpha}$. For example we get 
\begin{lemma}\label{lemma_gamma1/2}
    For any dimension $n$, the Hermitian metric $g(\alpha,-\frac{1}{2}\alpha)$ on the n-dimensional linear Hopf manifold is Bismut-Griffiths-non-negative; in particular, given any $\xi,\nu\in T^{1,0}M$, we have the following equation
    \begin{equation*}
        \Omega^{B}(\xi,\overline{\xi},\nu,\overline{\nu})=\frac{\alpha}{|z|^{8}}\left|(\xi\cdot\nu)|z|^{2}-(\xi\cdot z)(z\cdot\nu)\right|^{2}\geq 0 \;.
    \end{equation*}
\end{lemma}
\begin{rmk*}
    Notice that in any point $z\in M$ we will never get $\Omega^{B}(z)>0$ for metrics $g(\alpha,\beta)$, since both the terms $U_{\alpha}$ and $U_{\beta}$ vanish if $\xi=\lambda z$ or $\nu=\lambda z$ for $\lambda\in\mathbb{C}$.
\end{rmk*}
\begin{prop}\label{Lemma_gamma-n}
    There are coefficients $\gamma_{n}$ depending on the dimension $n$ of the linear Hopf manifold such that for all $\alpha>0$, $g(\alpha,\gamma\alpha)$ has Bismut-Griffiths-non-negative curvature tensor if and only if $\gamma\leq\gamma_{n}$. This coefficients are $\gamma_{2}=0$ and $\gamma_{n}=-\frac{1}{2}$ for $n\geq 3$.
\end{prop}
\begin{proof}
    We already know that $\gamma_{2}=0$ has the above property.\\
    Now suppose that $n\geq 3$. Take $\alpha>0$ and $\varepsilon>0$; by Lemma \ref{lemma_gamma1/2} we know that the metric $g(\alpha,-\frac{1}{2}\alpha)$ is Bismut-Griffiths-non-negative and the metric $g(\alpha,(-\frac{1}{2}+\varepsilon)\alpha)$ has Bismut curvature tensor given by
    \begin{equation*}
        \Omega^{B}(\xi,\overline{\xi},\nu,\overline{\nu})=\frac{\alpha}{|z|^{8}}\left|(\xi\cdot\nu)|z|^{2}-(\xi\cdot z)(z\cdot\nu)\right|^{2}+2\varepsilon U_{\beta}(\xi,\overline{\xi},\nu,\overline{\nu})\;.
    \end{equation*}
    On a point $z\in M$ with two zero coordinates (say $k$ and $l$), by equation (\ref{eq_U_beta}) we get
    \begin{align*}
        \Omega^{B}(\partial_{k},\overline{\partial}_{k},\partial_{l},\overline{\partial}_{l})_{z}&=2\varepsilon U_{\beta}(\partial_{k},\overline{\partial}_{k},\partial_{l},\overline{\partial}_{l})_{z}=-\frac{2\varepsilon}{|z|^{4}}<0\;.
    \end{align*}
\end{proof}

\subsection{HCFs evolution in the canonical family}
First of all, we compute the terms $S$ and $Q$ of the HCFs in the explicit case of a linear Hopf manifold equipped with $g(\alpha,\beta)$. By direct computations we get the Christoffel symbols of the Chern connection
\begin{equation*}
	\Gamma_{ij}^{k}=\frac{1}{|z|^{2}}\left(\frac{\beta}{\alpha}\delta_{i}^{k}\overline{z}_{j}-\delta_{j}^{k}\overline{z}_{i}\right)-\frac{\beta}{\alpha}\frac{\overline{z}_{i}\overline{z}_{j}z^{k}}{|z|^{4}}\;,
\end{equation*}
and the Chern curvatures
\begin{multline*}
	\Omega_{i\overline{j}k}^{l}=\frac{1}{|z|^{2}}\left[\delta_{k}^{l}\left(\delta_{ij}-\frac{\overline{z}_{i}z_{j}}{|z|^{2}}\right)-\frac{\beta}{\alpha}\delta_{i}^{l}\left(\delta_{jk}-\frac{\overline{z}_{k}z_{j}}{|z|^{2}}\right) \right. \\  \left. +\frac{\beta}{\alpha}\frac{(\delta_{jk}\overline{z}_{i}+\delta_{ij}\overline{z}_{k})|z|^{2}-2\overline{z}_{i}z_{j}\overline{z}_{k}}{|z|^{4}}z^{l}\right]\;;
\end{multline*}
\begin{equation*}
	\Theta^{(2)}_{i\overline{j}}=\frac{1}{|z|^{2}}\left[\left(n-1-\frac{\beta}{\alpha}\right)\delta_{ij}+\frac{\beta}{\alpha}\left(2n-1+\frac{\beta}{\alpha}(n-1)\right)\frac{\overline{z}_{i}z_{j}}{|z|^{2}}\right]\;.
\end{equation*}
The Chern torsion is $T_{ij}^{k}=\frac{1}{|z|^{2}}\left(\frac{\beta}{\alpha}+1\right)(\delta_{i}^{k}\overline{z}_{j}-\delta_{j}^{k}\overline{z}_{i})$ and we have the following quadratic terms in $T^{Ch}$:
\begin{align*}
	Q^{1}_{i\overline{j}}&=\frac{1}{|z|^{2}}\left(\frac{\beta}{\alpha}+1\right)^{2}\left[\frac{\alpha}{\alpha+\beta}\delta_{ij}+\left(n-2+\frac{\beta}{\alpha+\beta}\right)\frac{\overline{z}_{i}z_{j}}{|z|^{2}}\right]\;;\\
    Q^{2}_{i\overline{j}}&=\frac{2}{|z|^{2}}\left(\frac{\beta}{\alpha}+1\right)^{2}\frac{\alpha}{\alpha+\beta}\left[\delta_{ij}-\frac{\overline{z}_{i}z_{j}}{|z|^{2}}\right]\;;\\
    Q^{3}_{i\overline{j}}&=(n-1)^{2}\left(\frac{\beta}{\alpha}+1\right)^{2}\frac{\overline{z}_{i}z_{j}}{|z|^{4}}\;;\\
    Q^{4}_{i\overline{j}}&=\frac{1}{|z|^{2}}\left(\frac{\beta}{\alpha}+1\right)^{2}\frac{\alpha}{\alpha+\beta}(n-1)\left[\delta_{ij}-\frac{\overline{z}_{i}z_{j}}{|z|^{2}}\right]   \;.
\end{align*}

\begin{prop}\label{pr_general ODE}
    Let $M$ be a linear Hopf manifold equipped with the Hermitian metric $g(\alpha_{0},\beta_{0})$, then the HCF $\dot{g}=-S+aQ^{1}+bQ^{2}+cQ^{3}+dQ^{4}$ starting from $g(\alpha_{0},\beta_{0})$ evolves the metric as
    \begin{equation*}
	g(t)_{i\overline{j}}=\alpha(t)\frac{\delta_{ij}}{|z|^{2}}+\beta(t)\frac{\overline{z}_{i}z_{j}}{|z|^{4}}
\end{equation*} for $t\geq 0$, where $\alpha$ and $\beta$ satisfy the ODE system
\begin{equation}\label{general ODE}
    \begin{cases}
        \alpha(0)=\alpha_{0},\hspace{1em} \beta(0)=\beta_{0}\\
        \dot{\alpha}(t)=\frac{\beta}{\alpha}+1-n+\left(\frac{\beta}{\alpha}+1\right)(a+2b+(n-1)d)\\
        \dot{\beta}(t)=\frac{\beta}{\alpha}\left(1-2n-\frac{\beta}{\alpha}(n-1)\right)+\left(\frac{\beta}{\alpha}+1\right)^{2}(n-1)(a+(n-1)c)+\\
        \hspace{3.1em}-\left(\frac{\beta}{\alpha}+1\right)(a+2b+(n-1)d)
    \end{cases}
\end{equation}
\end{prop}

Thus we now turn to the study of the behaviour of the Bismut-Griffiths-non-negativity on linear Hopf manifolds equipped with $g(\alpha,\beta)$ metrics under the action of the HCFs. Since the Bismut-Griffiths-non-negativity of the metric $g(\alpha,\beta)$ only depends on the ratio $\gamma$, we give a useful result (Theorem \ref{th_stability}) which describes the evolution of $\gamma$ through the action of the HCFs.

Notice that in the ODE (\ref{general ODE}) both $\dot{\alpha}$ and $\dot{\beta}$ depends only on $\gamma=\frac{\beta}{\alpha}$. Thus we can rewrite the ODE system as
\begin{equation*}
    \begin{cases}
        \alpha(0)=\alpha_{0},\hspace{1em} \beta(0)=\beta_{0}\\
        \dot{\alpha}(t)=\gamma+1-n+\left(\gamma+1\right)(a+2b+(n-1)d)\\
        \dot{\beta}(t)=\gamma\left(1-2n-\gamma(n-1)\right)+\left(\gamma+1\right)^{2}(n-1)(a+(n-1)c)+\\
        \hspace{3.1em}-\left(\gamma+1\right)(a+2b+(n-1)d)
    \end{cases}
\end{equation*}
Then the ratio $\gamma$ evolves as
\begin{equation}\label{ODE of gamma}
    \begin{cases}
        \gamma(0)=\frac{\beta_{0}}{\alpha_{0}}\\
        \dot{\gamma}=\frac{1}{\alpha}(\gamma+1)\left[(F-n)\gamma + F\right]\\
        \dot{\alpha}=(\gamma+1)L-n
    \end{cases}
\end{equation}
where $F(a,b,c,d,n)=(n-2)a-2b+(n-1)^{2}c-(n-1)d$ and\\ $L(a,b,c,d)=1+a+2b+(n-1)d$.\\
Recall that $\beta>-\alpha$, so $\gamma>-1$. Thus $\gamma$ can not be equal to $-1$ and it can be equal to $\frac{F}{n-F}$ only if $F<n$. If we start with $\alpha_{0},\beta_{0}$ with ratio $\frac{F}{n-F}$, then $\alpha$ and $\beta$ will evolve as two straight lines, meaning that when $\gamma_{0}=\frac{F}{n-F}$ the flow acts on the metric by homotheties: these metrics are called static. Moreover, the following result shows that these static metrics are globally stable for the HCFs among the $g(\alpha,\beta)$ metrics.
\begin{theorem}\label{th_stability}
Consider an $n$-dimensional linear Hopf manifold and the HCF $\dot{g}=-S+aQ^{1}+bQ^{2}+cQ^{3}+dQ^{4}$. Suppose that the coefficients $(a,b,c,d)$ are such that $F(a,b,c,d,n)<n$, then the metric $g(1,\frac{F}{n-F})$ (as well as any of its multiples) is static for the flow. Moreover, any metric $g(\alpha_{0},\beta_{0})$ evolves along the flow so that the ratio $\gamma$ converges to $\frac{F}{n-F}$.
\end{theorem}
\begin{proof}
First of all, notice that the flows act by homothety on the static metrics, thus they are exactly those for which
\begin{equation*}
    \dot{\gamma}=\frac{1}{\alpha}(\gamma+1)\left[(F-n)\gamma + F\right]=0\; .
\end{equation*}
Since $\gamma>-1$, this could be the case only if $F<n$ and we get $\gamma=\frac{F}{n-F}$.\\
Now suppose that the starting metric $g(\alpha_{0},\beta_{0})$ has ratio $\gamma_{0}<\frac{F}{n-F}$. By the evolution equation (\ref{ODE of gamma}) for $\gamma$ we know that $\gamma$ is strictly increasing along the flow, moreover, it is bounded above from $\frac{F}{n-F}$. We now distinguish two cases, depending on $L$: when $\alpha$ is decreasing along the flow, meaning that $(\frac{F}{n-F}+1)L-n\leq 0$, and when it is not.\\
In the first case, suppose that $\gamma$ does not converge to $\frac{F}{n-F}$, then it needs to converge to some $\gamma_{\infty}$ with $\gamma_{0}<\gamma_{\infty}<\frac{F}{n-F}$. In this way
\begin{equation*}
    \dot{\alpha}<(\gamma_{\infty}+1)L-n<\left(\frac{F}{n-F}+1\right)L-n\leq 0\;,
\end{equation*} 
thus $\dot{\alpha}$ is uniformly strictly negative and so $\alpha$ will get to zero in finite time, say $T$; at the same time $T$, $\gamma$ will be increasing with infinite speed (by eq. (\ref{ODE of gamma})), which is a contradiction to the convergence $\gamma\rightarrow\gamma_{\infty}$.\\
In the second case, namely $\left(\frac{F}{n-F}+1\right)L-n> 0$, we can suppose without loss of generality that $(\gamma_{0}+1)L-n> 0$. Moreover,  since the term $(\gamma+1)$ in the formula of $\dot{\gamma}$ is positive and increasing we can suppress it and prove the convergence of $\gamma$ to $\frac{F}{n-F}$ with evolution equation
\begin{equation*}
    \dot{\gamma}=\frac{1}{\alpha}\left[(F-n)\gamma + F\right]\;.
\end{equation*}
Since $\gamma$ is bounded from above, then also $\dot{\alpha}$ is so. This means that we can bound $\alpha$ above with a straight line with positive slope $\alpha\leq\alpha(0)+At$. Thus finally we have $(\alpha(0)+At)\dot{\gamma}(t)\geq (F-n)\gamma(t)+F$. We have an explicit solution for this ODE
\begin{equation*}
\gamma(t)\geq C(At+\alpha(0))^{\frac{F-n}{A}}+\frac{F}{n-F}\;,
\end{equation*}
where the constant $C$ depends on the initial value $\gamma_{0}$. Since the exponent $\frac{F-n}{A}$ is negative we get the convergence to $\frac{F}{n-F}$ for $t\rightarrow\infty$.\\
A similar argument holds true also in the opposite case, namely if $\gamma_{0}>\frac{F}{n-F}$.
\end{proof}

This is another evidence of stability results as in \cite{Streets_2011}. Here we have global stability in the non-KE setting; however, this is an extremely particular case since it refers to linear Hopf manifolds equipped with a metric $g(\alpha,\beta)$.\\

\subsection{Bismut-Griffiths-non-negativity under the action of HCFs}
In this section we detect a subset of HCF of flows which preserve Bismut-Griffiths-non-negativity on linear Hopf manifolds equipped with a $g(\alpha,\beta)$ metric (see Theorem \ref{Th inequality all dimensions}). This subfamily is prescribed by inequalities of the coefficients $(a,b,c,d)$ characterizing the HCFs which depend on the dimension $n$. Take $\gamma_{n}$ as in Proposition \ref{Lemma_gamma-n}.
\begin{theorem}\label{Th inequality all dimensions}
    Consider an $n$-dimensional linear Hopf manifold equipped with a metric $g(\alpha_{0},\beta_{0})$, and suppose that $(n-2)a-2b+(n-1)^{2}c-(n-1)d\leq n\frac{\gamma_{n}}{\gamma_{n}+1}$. Then if the metric $g(\alpha_{0},\beta_{0})$ is Bismut-Griffiths-non-negative, the HCF with coefficients $(a,b,c,d)$ starting at $g(\alpha_{0},\beta_{0})$ preserves the Bismut-Griffiths-non-negativity. 
\end{theorem}
\begin{proof}
    Notice that since the metric $g(\alpha_{0},\beta_{0})$ is Bismut-Griffiths-non-negative, the initial ratio $\gamma_{0}$ must be $\gamma_{0}\leq\gamma_{n}$. Moreover, we have that
    \begin{equation*}
        F(a,b,c,d,n)=(n-2)a-2b+(n-1)^{2}c-(n-1)d\leq n\frac{\gamma_{n}}{\gamma_{n}+1}\leq 0<n
    \end{equation*}
    thus by Theorem \ref{th_stability} the ratio $\gamma$ will evolve along the flow converging to a value $\gamma_{\infty}=\frac{F}{n-F}\leq\gamma_{n}$. This means that the metric will remain Bismut-Griffiths-non-negative along the flow.
\end{proof}
On the other hand, when the above inequality is not satisfied, the flow does not preserve Bismut-Griffiths-non-negativity. More precisely
\begin{prop}\label{pr_sharpness}
    On linear Hopf manifolds of dimension $n$ the HCFs with coefficients $(a,b,c,d)$ such that $(n-2)a-2b+(n-1)^{2}c-(n-1)d>n\frac{\gamma_{n}}{\gamma_{n}+1}$ do not preserve Bismut-Griffiths-non-negativity. 
\end{prop}
\begin{proof}
    To prove the statement we suppose to perform the HCF with coefficients $(a,b,c,d)$ starting from the metric $g(1,\gamma_{n})$. By hypothesis this is Bismut-Griffiths-non-negative and $\gamma_{n}$ is the largest ratio for which this happens; thus we just have to verify that $\dot{\gamma}(0)>0$.\\
    In case $F\geq n$, we get that $(F-n)\gamma+F>0$ for any $\gamma$ since $\gamma>-1$. Otherwise, if 
    \begin{equation*}
        n>F(a,b,c,d,n)=(n-2)a-2b+(n-1)^{2}c-(n-1)d>n\frac{\gamma_{n}}{\gamma_{n}+1}
    \end{equation*}
    then $(F-n)\gamma_{n}+F>0$. Thus, if the inequality in the statement is satisfied, using the evolution equation (\ref{ODE of gamma}) for $\gamma$ we get that
    \begin{equation*}
        \dot{\gamma}(0)=\frac{1}{\alpha}(\gamma_{n}+1)\left[(F-n)\gamma_{n}+F\right]>0
    \end{equation*}
    and the thesis is proved.
\end{proof}
\begin{rmk*}
    Proposition \ref{pr_sharpness} shows that the inequalities of Theorem \ref{Th inequality all dimensions} are sharp, meaning that they detect the largest set of HCFs which preserve Bismut-Griffiths-non-negativity on linear Hopf manifolds equipped with metrics $g(\alpha,\beta)$.
\end{rmk*}

\subsection{Interesting flows}\label{sec_interesting flows}
We can use the above set of inequalities to check if some interesting flow preserve Bismut-Griffiths-non-negativity on linear Hopf manifolds with metrics $g(\alpha,\beta)$.\\
The Gradient flow of Streets and Tian, that is the one with coefficients $a=\frac{1}{2},b=-\frac{1}{4},c=-\frac{1}{2},d=1$, has $\frac{F}{n-F}=-\frac{n-1}{n+1}$; hence, $\frac{F}{n-F}\leq -\frac{1}{2} = \gamma_{n}$ for any $n>2$ and $\frac{F}{n-F}=-\frac{1}{3}$ when $n=2$. 
On the other hand, for the Ustinovskiy flow $F=1>0$ regardless of the dimension.\\
Beside these two flows we are interested in the \textit{pluriclosed flow}, introduced by Streets and Tian in \cite{Streets_2010}. They identified a particular choice of $Q$ which yields a flow that preserves the pluriclosed condition and so has a natural link with the Bismut connection. Specifically, in our notation $Q$ is identified by $a=1, b=c=d=0$. With this coefficients we get $F=n-2$. \\
Comparing these values with the inequalities in the previous results we have
\begin{prop*}
    Consider a linear Hopf manifold (of any dimension) equipped with a Bismut-Griffiths-non-negative metrics $g(\alpha_{0},\beta_{0})$, then
    \begin{itemize}
        \item the gradient flow of Streets and Tian starting at $g(\alpha_{0},\beta_{0})$ evolves preserving the Bismut-Griffiths-non-negativity;
        \item the Ustinovskiy flow starting at $g(\alpha_{0},\beta_{0})$ does not preserve the Bismut-Griffiths-non-negativity;
        \item unless $n=2$, the pluriclosed flow starting at $g(\alpha_{0},\beta_{0})$ does not preserve the Bismut-Griffiths-non-negativity.
    \end{itemize}
\end{prop*}
\begin{rmk*}
    We saw that the pluriclosed flow performed on linear Hopf surfaces with metrics $g(\alpha,\beta)$ preserves the Bismut-Griffiths-non-negativity. As a matter of fact, a metric $g(\alpha,\beta)$ on the linear Hopf manifold is pluriclosed if and only if $n=2$; thus it is interesting that the pluriclosed flow behave well with Bismut-Griffiths-non-negativity only in dimension two.
\end{rmk*}  

\section{Appendix. Computations on six-dimensional Calabi-Yau solvmanifolds}
We collect here the computations on six-dimensional Calabi-Yau solvmanifolds that lead to Theorems \ref{th_cplx solvmanifolds}, \ref{prop_positivity solvmanifolds} and \ref{th_HCF on solvmanifolds}.\\
Some of the following computation were performed with the help of the symbolic computation software Sage \cite{sage}.
\subsection{Nilmanifolds}
\subsubsection{Holomorphically-parallelizable nilmanifolds in Family (Np)}
Consider six-dimensional holomorphically-parallelizable nilmanifolds, i.e. nilmanifolds with holomorphically trivial tangent bundle. On these nilmanifolds the complex structure equations are
\begin{center}
    $d\varphi^{1}=d\varphi^{2}=0,\hspace{0.2em} d\varphi^{3}=\rho\varphi^{12}; \hspace{1em} \rho=0,1$
\end{center}
The case $\rho=0$ refers to the Torus which is K{\"a}hler and flat; thus we consider only the case of Iwasawa manifold ($\rho=1$).\\
A direct calculation shows that $\Omega_{ij\cdot\cdot}= \Omega_{\cdot\cdot kl}=0$ for any $i,j,k,l\in\{1,2,3\}$.\\
Moreover, we have computed the following determinant and coefficient: 
\begin{equation*}
    \Omega_{1\overline{1}3\overline{3}}\Omega_{2\overline{2}3\overline{3}}-\Omega_{1\overline{2}3\overline{3}}\Omega_{2\overline{1}3\overline{3}}=-\frac{t^{10}}{32\sqrt{-1}\det\Xi}
\end{equation*}
\begin{equation*}
     \Omega_{1\overline{1}3\overline{3}}=\frac{t^{4}(r^{2}t^{2}-|z|^{2})}{16\sqrt{-1}\det\Xi}
\end{equation*}
Showing that this Bismut curvature tensor is neither non-positive nor non-negative.

\subsubsection{Nilmanifolds in Family (Ni)}
Consider the generic Hermitian structure of this family
\begin{equation*}
    d\varphi^{1}=d\varphi^{2}=0;\hspace{0.4em} d\varphi^{3}=\rho \varphi^{12}+\varphi^{1\overline{1}}+\lambda\varphi^{1\overline{2}}+D\varphi^{2\overline{2}}
\end{equation*}where $\rho\in\{0,1\}$, $\lambda\geq 0$, $Im D\geq 0$.\\
According to equations (2.4–2.5) of \cite{MR3334093}, up to linear biholomorphism we can take $v=z=0$ and $r^{2}=1$ in the generic expression (\ref{meetric on nilmanifolds}):
\begin{equation*}
    2\omega=\sqrt{-1}(\varphi^{1\overline{1}}+s^{2}\varphi^{2\overline{2}}+t^{2}\varphi^{3\overline{3}})+u\varphi^{1\overline{2}}-\overline{u}\varphi^{2\overline{1}}
\end{equation*}  
Let us take the element $\Omega_{231\overline{3}}= \frac{\rho s^{2}t^{2}}{16\sqrt{-1}\det\Xi}$. It vanishes if and only if $\rho=0$.\\ Taking into account the classification of complex structures up to equivalence (see \cite{article}) we set the coefficients $\rho,\lambda$ and $D$ (and the Lie algebras) as follows: 
\begin{itemize}
    \item $(\rho,\lambda,D)=(0,0,\sqrt{-1})$, Lie algebra $\mathfrak{h}_{2}$;
    \item $(\rho,\lambda,D)=(0,0,\pm 1)$, Lie algebra $\mathfrak{h}_{3}$;
    \item $(\rho,\lambda,D)=(0,1,\frac{1}{4})$, Lie algebra $\mathfrak{h}_{4}$;
    \item $(\rho,\lambda)=(0,1)$ and $D\in[0,\frac{1}{4})$, Lie algebra $\mathfrak{h}_{5}$;
    \item $(\rho,\lambda,D)=(0,0,0)$, Lie algebra $\mathfrak{h}_{8}$.
\end{itemize}
In any of these cases the computation of the curvature elements yields that, for any $i,j,k,l\in\{1,2,3\}$ $\Omega_{ij\cdot\cdot}=\Omega_{\cdot\cdot kl}=0$.\\ 
Suppose $\lambda=0$ (thus Lie algebras $\mathfrak{h}_{2},\mathfrak{h}_{3}$ and $\mathfrak{h}_{8}$). From direct computations we get the following elements of the Bismut curvature tensor
\begin{align*}
    &\Omega_{1\overline{1}1\overline{1}}= t^{2}
    &&\Omega_{1\overline{1}2\overline{2}}= Re(D)t^{2}\\
    &\Omega_{2\overline{2}1\overline{1}}= Re(D)t^{2}
    &&\Omega_{2\overline{2}2\overline{2}}= |D|^{2}t^{2}\\
    &\Omega_{3\overline{3}1\overline{2}}= -\frac{Re(\sqrt{-1}D)}{(s^{2} - |u|^{2})}t^{4}u
    &&\Omega_{3\overline{3}2\overline{1}}= -\frac{Re(\sqrt{-1}D)}{(s^{2} - |u|^{2})}t^{4}\overline{u}\\
\end{align*}
Thus if $D=-1$ ($\mathfrak{h}_{3}$), then $\Omega_{1\overline{1}2\overline{2}}<0<\Omega_{1\overline{1}1\overline{1}}$ and the curvature tensor is neither non-negative nor non-positive. On the other hand, if $D=\sqrt{-1}$ ($\mathfrak{h}_{2}$), then the determinant
\begin{equation*}
    \Omega_{3\overline{3}1\overline{1}}\Omega_{3\overline{3}2\overline{2}}-\Omega_{3\overline{3}1\overline{2}}\Omega_{3\overline{3}2\overline{1}}=-\frac{t^{8}|u|^{2}}{(s^{2}-|u|^{2})^{2}}\leq 0\;.
\end{equation*}
Thus the Bismut curvature tensor is non-negative if and only if $u=0$. Finally, for $D=1$ or $D=0$ ($\mathfrak{h}_{3}$ or $\mathfrak{h}_{8}$) we have Bismut-Griffiths non-negativity.
\begin{rmk}\label{rmk: 1}
    Suppose we are in case of Lie algebra $\mathfrak{h}_{2}$, hence with coefficients $(\rho,\lambda,D)=(0,0,\sqrt{-1})$. Suppose also that $u=0$ (i.e. the metric $g$ is diagonal, since we are supposing $v=z=0$) then also $S$ and $Q$ are diagonal. This means that any HCF preserves the condition $u=v=z=0$.
\end{rmk}
Now we turn to the Lie algebras $\mathfrak{h}_{4}$ and $\mathfrak{h}_{5}$, for which we have computed the following element and the determinant of the curvature tensor:
\begin{equation*}
        \Omega_{1\overline{1}1\overline{1}}=t^{2}>0\;,
\end{equation*}
\begin{equation*}
    \Omega_{1\overline{1}1\overline{1}}\Omega_{1\overline{1}2\overline{2}}-\Omega_{1\overline{1}1\overline{2}}\Omega_{1\overline{1}2\overline{1}}=t^{4}(D-\frac{1}{4})\;.
\end{equation*}
Thus, in case of Lie algebra $\mathfrak{h}_{5}$ (i.e. $D<\frac{1}{4}$) the Bismut curvature tensor is neither non-positive nor non-negative. Now, if $D=\frac{1}{4}$, we compute the second Ricci tensors of The Bismut curvature $Ric2^{B}$. We have that
\begin{equation*}
    Ric2^{B}_{1\overline{1}}Ric2^{B}_{2\overline{2}}-Ric2^{B}_{1\overline{2}}Ric2^{B}_{2\overline{1}}=-\frac{|4s^{2}-4\sqrt{-1}\overline{u}+1|^{2}}{16(s^{2}-|u|^{2})^{2}}<0\;.
\end{equation*}

\subsubsection{Nilmanifolds in Family (Nii)} \label{Nii}
Consider the complex structure equations
\begin{equation*}
    d\varphi^{1}=0, \hspace{0.4em} d\varphi^{2}=\varphi^{1\overline{1}}, \hspace{0.4em} d\varphi^{3}=\rho\varphi^{12}+B\varphi^{1\overline{2}}+c\varphi^{2\overline{1}}
\end{equation*} 
where $\rho\in\{0,1\}$, $c\geq 0$, $B\in\mathbb{C}$ satisfying $(\rho,B,c)\neq (0,0,0)$.\\
If $\rho=0$ we have the coefficients 
\begin{align*}
    &\Omega_{231\overline{2}}=\frac{c(s^{2}t^{2}-|v|^{2})^{2}}{16\sqrt{-1}\det\Xi}, &&\Omega_{232\overline{1}}=\frac{-\overline{B}(s^{2}t^{2}-|v|^{2})^{2}}{16\sqrt{-1}\det\Xi}
\end{align*} 
These vanish only if $c=B=0$ which is impossible since $(\rho,B,c)\neq 0$. Thus we take $\rho=1$ and compute
\begin{equation*}
    \Omega_{232\overline{3}}=\frac{s^{2}t^{2} - |v|^{2}}{16\sqrt{-1}\det\Xi}t^{4}\overline{B}\;.
\end{equation*}
Hence also $B=0$. \\
Now we prove that $c=v=0$. First of all, if $c=0$, we have
\begin{equation*}
    \Omega_{231\overline{2}}=\frac{s^{2}t^{2}-|v|^{2}}{16\sqrt{-1}\det\Xi}\overline{v}^{2}\;,
\end{equation*}
which implies $v=0$; on the other hand, if $v=0$, we have
\begin{equation*}
    \Omega_{131\overline{3}}=ct^{4}\frac{r^{2}t^{2}-|z|^{2}}{16\sqrt{-1}\det\Xi}\;,
\end{equation*}
which implies $c=0$. Thus $c=0$ if and only if $v=0$. Suppose $c\neq 0$ (hence $v\neq 0$), then we compute the following elements of the Bismut curvature tensor:
\begin{multline*}
    \Omega_{131\overline{1}}=\left[\sqrt{-1}ct^{2}(r^{2}t^{2}\overline{z}+\sqrt{-1}t^{2}|u|^{2}+\overline{uv}z-uv\overline{z}-\overline{z}|z|^{2})\right. \\  
    \left. -\sqrt{-1}t^{2}u\overline{vz}-(cv+\overline{v})|z|^{2}\overline{v}\right]/16\sqrt{-1}\det\Xi
\end{multline*}
\begin{multline*}
    \Omega_{131\overline{2}}=\left[\sqrt{-1}ct^{2}(s^{2}t^{2}u-u|v|^{2}-\overline{v}|z|^{2}+r^{2}t^{2}\overline{v}-\sqrt{-1}s^{2}\overline{v}z)\right. \\
    \left. -\sqrt{-1}t^{2}u\overline{v}^{2}-(cv+\overline{v})\overline{v}^{2}z\right]/16\sqrt{-1}\det\Xi
\end{multline*}
\begin{multline*}
    \Omega_{121\overline{3}}=\left[ct^{2}(-8\sqrt{-1}\det\Xi+r^{2}|v|^{2}-\sqrt{-1}s^{4}z+s^{2}uv-\sqrt{-1}\overline{uv}z)\right.\\
    \left. +(\sqrt{-1}s^{2}\overline{v}z-u|v|^{2})(cv+\overline{v})\right]/8\sqrt{-1}\det\Xi
\end{multline*}
\begin{align*}
    \Omega_{121\overline{2}}=&t^{2}\frac{\sqrt{-1}u|v|^{2}+s^{2}\overline{v}z-ct^{2}\overline{u}z-\sqrt{-1}cr^{2}t^{2}v}{8\sqrt{-1}\det\Xi}
\end{align*}
From $\Omega_{131\overline{1}}=\Omega_{131\overline{2}}=0$ we get:
\begin{align*}
    \left[\sqrt{-1}ct^{2}(r^{2}t^{2}-|z|^{2})-(\sqrt{-1}t^{2}u+\overline{vz})(cv+\overline{v})\right]\overline{z}&=ct^{2}\overline{u}(t^{2}u-\sqrt{-1}\overline{v}z)\\
    \left[\sqrt{-1}ct^{2}(r^{2}t^{2}-|z|^{2})-(\sqrt{-1}t^{2}u+\overline{vz})(cv+\overline{v})\right]\overline{v}&=-ct^{2}s^{2}(\sqrt{-1}t^{2}u+\overline{v}z)
\end{align*}
Hence 
\begin{equation*}
    (\sqrt{-1}t^{2}u+\overline{v}z)(\overline{uv}+\sqrt{-1}s^{2}\overline{z})=0\;.
\end{equation*}
Notice that from this equations we also get that $u=0$ if and only if $z=0$; however they can not vanish or we would get $cv=0$ from $\Omega_{131\overline{2}}=0$. Thus $u,z$ ($v,c$) are different from zero and we distinguish two cases: $\sqrt{-1}t^{2}u+\overline{v}z=0$ and $\overline{uv}+\sqrt{-1}s^{2}\overline{z}=0$.
In the first case, we have
\begin{align*}
    0&=\sqrt{-1}ct^{2}(r^{2}t^{2}-|z|^{2})-(\sqrt{-1}t^{2}u+\overline{vz})(cv+\overline{v})\\
    &=\sqrt{-1}ct^{2}(r^{2}t^{2}-|z|^{2})\;,
\end{align*}
thus $c=0$, which is a contradiction.
In the second case, $\Omega_{121\overline{2}}=0$ and $\overline{uv}+\sqrt{-1}s^{2}\overline{z}=0$ imply
\begin{align*}
    0&=\sqrt{-1}u|v|^{2}+s^{2}\overline{v}z=ct^{2}(\overline{u}z+\sqrt{-1}r^{2}v)\;.
\end{align*}
Multiplying by $\overline{v}$ and using again $\overline{uv}+\sqrt{-1}s^{2}\overline{z}=0$ we obtain $s^{2}|z|^{2}=r^{2}|v|^{2}$. Finally, with these equations $\Omega_{121\overline{3}}$ become
\begin{equation*}
    \Omega_{121\overline{3}}=-ct^{2}\neq 0\;.
\end{equation*}
This shows that $v=c=0$ is needed to satisfy (Cplx).\\
With these parameters ($\rho=1; c=B=0$) and $v=0$ (Cplx) is satisfied and we get the following element and determinant of the Bismut curvature tensor:
\begin{align*}
    &\Omega_{2\overline{2}3\overline{3}}=\frac{s^{2}t^{6}}{16\sqrt{-1}\det\Xi}\\
    &\Omega_{2\overline{2}3\overline{3}}\Omega_{1\overline{1}3\overline{3}}-\Omega_{1\overline{2}3\overline{3}}\Omega_{2\overline{1}3\overline{3}}=-\frac{t^{10}}{32\sqrt{-1}\det\Xi}
\end{align*}
showing that the curvature tensor is neither non-negative nor non-positive.
\begin{rmk}\label{rmk: 2}
    With parameters ($\rho=1; c=B=0$) the condition $v=0$ (and hence (Cplx)) is preserved by any Hermitian curvature flow in the family HCF.
\end{rmk}

\subsubsection{Nilmanifolds in Family (Niii)}
Consider the complex structure equations 
\begin{equation*}
    d\varphi^{1}=0,\hspace{0.4em} d\varphi^{2}=\varphi^{13}+\varphi^{1\overline{3}},\hspace{0.4em} d\varphi^{3}=\sqrt{-1}\rho\varphi^{1\overline{1}}+\delta\sqrt{-1}(\varphi^{1\overline{2}}-\varphi^{2\overline{1}})
\end{equation*}
where $\rho\in\{0,1\}$ and $\delta=\pm 1$. From a direct computation we get the following elements of the Bismut curvature tensor:
\begin{align*}
    \Omega_{122\overline{3}}&=-\frac{(\sqrt{-1}uv + s^{2}z)v^{2}}{16\sqrt{-1}\det\Xi}\\
    \Omega_{121\overline{2}}&=-\frac{(\sqrt{-1}\delta\rho s^{2}z-\sqrt{-1}r^{2}v - \delta\rho uv -\overline{u}z)s^{2}v}{16\sqrt{-1}\det\Xi}\\
    \Omega_{132\overline{2}}&=\frac{(\sqrt{-1}t^{2}u+z\overline{v})s^{2}v}{16\sqrt{-1}\det\Xi}\\
    \Omega_{122\overline{1}}&=\frac{\sqrt{-1}s^{4}z^{2}-2s^{2}uvz+s^{2}\overline{u}vz-\sqrt{-1}u^{2}v^{2}+\sqrt{-1}v^{2}|u|^{2}}{16\sqrt{-1}\det\Xi}
\end{align*}
First of all we prove that $u,v$ and $z$ must vanish: suppose $v\neq 0$, then imposing $\Omega_{122\overline{3}}=0$ we get $s^{2}z=-\sqrt{-1}uv$. Now $\Omega_{121\overline{2}}=0$ implies $r^{2}v=\sqrt{-1}\overline{u}z$, and $\Omega_{132\overline{2}}=0$ implies $t^{2}u=\sqrt{-1}\overline{v}z$. These three equations together would imply that $\det\Xi=0$ which is a contradiction, thus $v$ must vanish. Moreover, if $v=0$ from $\Omega_{122\overline{1}}=0$ we get also $z=0$. Finally, $\Omega_{133\overline{2}}$ with $v=z=0$ is
\begin{equation*}
    \Omega_{133\overline{2}}=-\frac{s^{2}t^{2}(\sqrt{-1}s^{2}-t^{2})u}{16\sqrt{-1}\det\Xi}\;.
\end{equation*}
Thus, also $u$ mast vanish.\\
Now, for $u=v=z=0$ we have $\Omega_{133\overline{1}}=\frac{1}{2}(\delta\sqrt{-1}t^{2}-s^{2})\neq 0$,
showing that (Cplx) is never satisfied.

\subsection{Solvmanifolds}
\subsubsection{Solvmanifolds in Family (Si)}
Consider the generic Hermitian structure of this family
\begin{equation*}
    d\varphi^{1}=A(\varphi^{13}+\varphi^{1\overline{3}}),\hspace{0.4em} d\varphi^{2}=-A(\varphi^{23}+\varphi^{2\overline{3}}),\hspace{0.4em} d\varphi^{3}=0
\end{equation*}
where $A=\cos{\theta}+\sqrt{-1}\sin{\theta}$ and $\theta\in[0,\pi)$.\\
We directly compute (and set equal to zero)
\begin{equation*}
    \Omega_{123\overline{3}}=-\sqrt{-1}|A|^{2}\frac{ r^{2}uv^{2} + s^{2}z^{2}\overline{u} }{8\sqrt{-1}\det\Xi}\;.
\end{equation*}
This vanishes only if $r^{2}uv^{2} + s^{2}z^{2}\overline{u}=0$ since $A\neq 0$. We compute the following coefficients of the Bismut curvature tensor (where $I=\sqrt{-1}$):
\begin{multline*}
    \Omega_{133\overline{1}}=-\left[4Ar^{2}t^{2}|u|^{2} + (A+\overline{A})r^{4}|v|^{2} -(\overline{A}+3A)Ir^{2}z\overline{u}\overline{v}  \right.\\
    \left. + (A+\overline{A})Ir^{2}uv\overline{z} -(A - \overline{A})|u|^{2}|z|^{2}\right]\overline{A}/16\sqrt{-1}\det\Xi
\end{multline*}
\begin{multline*}
    \Omega_{133\overline{2}}=-\left[-4IAr^{2}s^{2}t^{2}u -(A+\overline{A})r^{2}s^{2}z\overline{v}  + (A-\overline{A})Ir^{2}|v|^{2}u \right.\\
    \left. + (3A-\overline{A})Is^{2}|z|^{2}u - (A - \overline{A})u^{2}v\overline{z}\right]\overline{A}/16\sqrt{-1}\det\Xi
\end{multline*}
\begin{multline*}
    \Omega_{233\overline{1}}=\left[-4IAr^{2}s^{2}t^{2}\overline{u} + (A+\overline{A})r^{2}s^{2}v\overline{z} + (3A-\overline{A})Ir^{2}|v|^{2}\overline{u} \right.\\
    \left. + (A-\overline{A})Is^{2}|z|^{2}\overline{u} + (A - \overline{A})z\overline{u}^{2}\overline{v}\right]\overline{A}/16\sqrt{-1}\det\Xi
\end{multline*}
\begin{multline*}
    \Omega_{233\overline{2}}=-\left[4As^{2}t^{2}|u|^{2} + (A+\overline{A})s^{4}|z|^{2} -(A+\overline{A})Is^{2}z\overline{u}\overline{v} \right.\\
    \left. + (3A+\overline{A})Is^{2}uv\overline{z} - (A - \overline{A})|u|^{2}|v|^{2}\right]\overline{A}/16\sqrt{-1}\det\Xi
\end{multline*}
The system of equations generated by the vanishing of these four coefficients has $u=v=z=0$ as unique solution. The computations follow exactly the same structure as for solvmanifolds in family (Siv3), see \S\ref{susec_Siv3}. Moreover a direct computation shows that with this hypothesis (Cplx) is satisfied.
\begin{rmk}\label{rmk: 3}
    The invariant metric $g$ with $u=v=z=0$ is Chern-flat. Moreover, with these parameters also $Q$ is diagonal; hence (Cplx) is preserved by any HCF.     
\end{rmk}
We computed the following elements of the Bismut curvature tensor:
\begin{align*}
    &\Omega_{1\overline{1}1\overline{1}}=2Re(A)^{2}\frac{r^{4}}{t^{2}},
    &\Omega_{1\overline{1}3\overline{3}}=-2r^{2}Re(A)^{2},
\end{align*} showing that if $A\neq\sqrt{-1}$ it is neither non-negative nor non-positive. \\ 
For parameter $A=\sqrt{-1}$, corresponding to the Lie algebra $\mathfrak{g}_{2}^{0}$, the diagonal metrics are K{\"a}hler, hence, K{\"a}hler-flat (see the Remark above). By \cite{boothby1958}, the complex solvmanifold is in fact biholomorphic to a holomorphically-parallelizable manifold.

\subsubsection{Solvmanifolds in Family (Sii)}
Consider the complex structure equations (where $x\in\mathbb{R}_{>0}$)
\begin{align*}
    d\varphi^{1}&=0\\ d\varphi^{2}&=-\frac{1}{2}\varphi^{13}-(\frac{1}{2}+\sqrt{-1}x)\varphi^{1\overline{3}}+\sqrt{-1}x\varphi^{3\overline{1}},\\
    d\varphi^{3}&=\frac{1}{2}\varphi^{12}+(\frac{1}{2}-\frac{\sqrt{-1}}{4x})\varphi^{1\overline{2}}+\frac{\sqrt{-1}}{4x}\varphi^{2\overline{1}}
\end{align*}
Working on the elements $\Omega_{232\overline{3}}$ and $\Omega_{233\overline{3}}$ (which we set equal to zero) we get $s^{2}=t^{2}$, see \cite{angella2018gauduchon} for details. Then
\begin{equation*}
    \Omega_{121\overline{2}}=\frac{t^{2}(2x-\sqrt{-1})}{16x}\neq 0\;,
\end{equation*}
and so (Cplx) is never satisfied.

\subsubsection{Solvmanifolds in Families (Siii1), (Siii3), (Siii4)}
Recall that the Lie algebras underlying (Siii1), (Siii3), and (Siii4) are, respectively, $\mathfrak{g}_{4}$, $\mathfrak{g}_{6}$, and $\mathfrak{g}_{7}$. In order to give a unified argument, we will gather the complex structure equations as follows:
\begin{equation*}
    d\varphi^{1}=\sqrt{-1}(\varphi^{13}+\varphi^{1\overline{3}}),\hspace{0.4em} d\varphi^{2}=-\sqrt{-1}(\varphi^{23}+\varphi^{2\overline{3}}),\hspace{0.4em} d\varphi^{3}=x\varphi^{1\overline{1}}+y\varphi^{2\overline{2}}
\end{equation*}
where $(x,y)=(\pm 1,0)$ for $\mathfrak{g}_{4}$, $(x,y)=(1,1)$ for $\mathfrak{g}_{6}$ and $(x\,,\,y=-x)=(\pm 1,\mp 1)$ for $\mathfrak{g}_{7}$. In particular $x\neq 0$.\\
Imposing the symmetries (Cplx) on the Bismut curvature tensor, we get that $y$ must be zero (meaning that the underling Lie algebra is $\mathfrak{g}_{4}$) and the metric described by the equation (\ref{meetric on nilmanifolds}) need to satisfy $u=v=z=0$; see \cite{angella2018gauduchon} for details.\\
With these condition (Cplx) is satisfied and the only non-zero coefficients of type $\Omega_{i\overline{j}k\overline{l}}$ of the curvature tensor is $\Omega_{1\overline{1}1\overline{1}}=t^{2}$.
\begin{rmk}\label{rmk: 4}
    If $u=v=z=0$ (i.e. the metric $g$ is diagonal) then also $S$ and $Q$ are diagonal. This means that any HCF preserves the condition $u=v=z=0$.
\end{rmk}

\subsubsection{Solvmanifolds in Family (Siii2)}
The complex structure equations for this family are the following:
\begin{equation*}
    d\varphi^{1}=\varphi^{13}+\varphi^{1\overline{3}},\hspace{0.4em}
    d\varphi^{2}=-\varphi^{23}-\varphi^{2\overline{3}},\hspace{0.4em} 
    d\varphi^{3}=\varphi^{1\overline{2}}+\varphi^{2\overline{1}}
\end{equation*}
Imposing the symmetries (Cplx) on the Bismut curvature tensor, we get that the metric described by the equation (\ref{meetric on nilmanifolds}) need to satisfy $v=z=0$; see \cite{angella2018gauduchon} for details.\\
From a direct computation we get 
\begin{align*}
    &\Omega_{133\overline{1}}=\sqrt{-1}r^{2}t^{2}u\frac{t^{2}+2\sqrt{-1}\overline{u}}{8\sqrt{-1}\det\Xi},
    &&\Omega_{133\overline{2}}=r^{2}s^{2}t^{2}\frac{t^{2}+2\sqrt{-1}u}{8\sqrt{-1}\det\Xi}
\end{align*}
$\Omega_{133\overline{2}}=0$ implies $t^{2}+2\sqrt{-1}u=0$, and then $\Omega_{133\overline{1}}=0$ leads to $t^{2}+2\sqrt{-1}\overline{u}$. These two equations together imply that $u$ is real which is in contradiction with both of them. This shows that (Cplx) is never satisfied.

\subsubsection{Solvmanifolds in Families (Siv1)}
Consider the complex structure equations for this family:
\begin{equation*}
    d\varphi^{1}=-\varphi^{13},\hspace{0.4em}
    d\varphi^{2}=\varphi^{23},\hspace{0.4em} 
    d\varphi^{3}=0
\end{equation*}
A direct computation shows that $\Omega_{ij\cdot\cdot}= \Omega_{\cdot\cdot kl}=0$ for any $i,j,k,l\in\{1,2,3\}$.\\
Moreover, we have the following coefficient and determinant of the curvature tensor:
\begin{align*}
    &\Omega_{1\overline{1}1\overline{1}}=\frac{r^{2}s^{2}-|u|^{2}}{16\sqrt{-1}\det\Xi}r^{4},
    &&\Omega_{1\overline{1}1\overline{1}}\Omega_{1\overline{1}3\overline{3}}-\Omega_{1\overline{1}1\overline{3}}\Omega_{1\overline{1}3\overline{1}}=-\frac{r^{2}s^{2}-|u|^{2}}{32\sqrt{-1}\det\Xi}r^{6}
\end{align*}

\subsubsection{Solvmanifolds in Families (Siv2)}
Recall the complex structure equations for this family:
\begin{equation*}
    d\varphi^{1}=2\sqrt{-1}\varphi^{13}+\varphi^{3\overline{3}},\hspace{0.4em}
    d\varphi^{2}=-2\sqrt{-1}\varphi^{23}-x\varphi^{3\overline{3}},\hspace{0.4em} 
    d\varphi^{3}=0
\end{equation*} where $x=0,1$.\\
Consider the terms $\Omega_{123\overline{1}}$ and $\Omega_{123\overline{2}}$:
\begin{align*}
    \Omega_{123\overline{1}}&=-\frac{(r^{2}s^{2}-|u|^{2})(xr^{2}s^{2}+x|u|^{2}+2\sqrt{-1}r^{2}\overline{u})}{8 \det \Xi}\\
    \Omega_{123\overline{2}}&=-\frac{(r^{2}s^{2}-|u|^{2})(r^{2}s^{2}+|u|^{2}-2\sqrt{-1}xs^{2}u)}{8 \det \Xi}
\end{align*}
Notice that, $\Omega_{123\overline{1}}=\Omega_{123\overline{2}}=0$ if and only if
\begin{equation*}
    xr^{2}s^{2}+x|u|^{2}+2\sqrt{-1}r^{2}\overline{u}=r^{2}s^{2}+|u|^{2}-2\sqrt{-1}xs^{2}u=0\;.
\end{equation*}
If $x=1$, these equations imply $Re (u)=0$, $Im (u)=-r^{2}$ and $r^{2}=s^{2}$, which is a contradiction to the positive definiteness of the metric. Hence, $x=0$ and
$\Omega_{123\overline{2}}$ is always different from zero.

\subsubsection{Solvmanifolds in Families (Siv3)}\label{susec_Siv3}
The complex structure equations for this family are the following (with $A\in\mathbb{C}\setminus\mathbb{S}^{1}$):
\begin{equation*}
    d\varphi^{1}=A\varphi^{13}-\varphi^{1\overline{3}},\hspace{0.4em}
    d\varphi^{2}=-A\varphi^{23}+\varphi^{2\overline{3}},\hspace{0.4em} 
    d\varphi^{3}=0
\end{equation*} 
We directly compute (and set equal to zero)
\begin{equation*}
    \Omega_{123\overline{3}}=\frac{ \sqrt{-1}(r^{2}uv^{2} + s^{2}z^{2}\overline{u})A }{8\sqrt{-1}\det\Xi}\;.
\end{equation*}
This vanishes if $A=0$ or $r^{2}uv^{2} + s^{2}z^{2}\overline{u}=0$. We start analyzing the case $A\neq 0$. We compute the following coefficients of the Bismut curvature tensor (where $I=\sqrt{-1}$):
\begin{multline*}
    \Omega_{133\overline{1}}=\left[ 4Ar^{2}t^{2}|u|^{2} + (A-1)r^{4}|v|^{2} + (1-3A)Ir^{2}z\overline{u}\overline{v} \right.\\
    \left. + (A-1)Ir^{2}uv\overline{z} - (A+1)|u|^{2}|z|^{2}\right]16\sqrt{-1}\det\Xi
\end{multline*}
\begin{multline*}
    \Omega_{133\overline{2}}=\left[-4IAr^{2}s^{2}t^{2}u + (3A+1)Is^{2}u|z|^{2} + (1+A)Ir^{2}u|v|^{2} \right.\\
    \left.+ (1-A)r^{2}s^{2}z\overline{v} -(A+1)u^{2}v\overline{z} \right]/16\sqrt{-1}\det\Xi
\end{multline*}
\begin{multline*}
    \Omega_{233\overline{1}}=\left[4IAr^{2}s^{2}t^{2}\overline{u} + (1-A)r^{2}s^{2}v\overline{z} -(3A+1)Ir^{2}|v|^{2}\overline{u} \right.\\
    \left. -(A+1)Is^{2}|z|^{2}\overline{u} - (A + 1)z\overline{u}^{2}\overline{v}\right]/16\sqrt{-1}\det\Xi
\end{multline*}
\begin{multline*}
    \Omega_{233\overline{2}}=\left[ 4As^{2}t^{2}|u|^{2} + (A-1)s^{4}|z|^{2} +(1-A)Is^{2}z\overline{u}\overline{v} \right.\\ \left.+ (3A-1)Is^{2}uv\overline{z} - (A + 1)|u|^{2}|v|^{2} \right]/16\sqrt{-1}\det\Xi
\end{multline*}
Notice that $A-1\neq 0$ by hypothesis, thus if $u=0$ we get also $v=0$ and $z=0$ from $\Omega_{133\overline{1}}=0$ and $\Omega_{233\overline{2}}=0$ respectively. On the other hand, if $u\neq 0$ then $v$ vanishes if and only if $z$ vanishes (from $r^{2}uv^{2} + s^{2}z^{2}\overline{u}=0$), and they can not vanish together otherwise $u$ should also be $0$ (from $\Omega_{233\overline{2}}=0$). Now suppose $u,v,z\neq 0$ and consider the following equations:
\begin{align}
    \Omega_{133\overline{2}}-\overline{\Omega_{233\overline{1}}}&=\frac{s^{2}|z|^{2}-r^{2}|v|^{2}}{8\sqrt{-1}\det\Xi}A\sqrt{-1}=0 \nonumber \\
    \Omega_{133\overline{1}}|v|^{2}-\Omega_{233\overline{2}}|z|^{2}&=-\frac{\overline{uv}z+uv\overline{z}}{8\sqrt{-1}\det\Xi}A\sqrt{-1}r^{2}|v|^{2}=0 \nonumber \\
    \Omega_{233\overline{1}}uv-\Omega_{233\overline{2}}\overline{u}z&=A|u|^{2}\frac{s^{2}t^{2}(\sqrt{-1}r^{2}v-\overline{u}z)-2\sqrt{-1}r^{2}|v|^{2}v}{4\sqrt{-1}\det\Xi}=0 \label{uno} \\
    \Omega_{133\overline{1}}uv-\Omega_{133\overline{2}}\overline{u}z&=A|u|^{2}\frac{r^{2}t^{2}(uv-\sqrt{-1}s^{2}z)-2\sqrt{-1}r^{2}|v|^{2}z}{4\sqrt{-1}\det\Xi}=0 \label{due}
\end{align}
where we used the first one to get the second and the first two to get the last two. Finally from $(\ref{uno})\cdot z - (\ref{due})\cdot v=0$ we get $vz=0$ which is a contradiction. This shows that $u,v$ and $z$ must be zero and a direct computation shows that with this hypothesis (Cplx) is satisfied.\\
In case $A=0$, $\Omega_{123\overline{1}}$ and $\Omega_{123\overline{2}}$ become
\begin{align*}
    &\Omega_{123\overline{1}}=\frac{ (r^{2}s^{2}-|u|^{2})(\sqrt{-1}r^{2}v +z\overline{u}) }{16\sqrt{-1}\det\Xi},
    &&\Omega_{123\overline{2}}=\frac{(r^{2}s^{2}-|u|^{2})(uv-\sqrt{-1}s^{2}z) }{16\sqrt{-1}\det\Xi}
\end{align*}
The equations $\sqrt{-1}r^{2}v +z\overline{u}=0$ and $\sqrt{-1}s^{2}z -uv=0$ implies that $v$ vanishes if and only if $z$ vanishes. Moreover, if they are both different from zero, we can multiply the first one by $\overline{v}$ and the second one by $\overline{z}$; this leads to $\sqrt{-1}r^{2}|v|^{2}+\overline{uv}z=0=\sqrt{-1}s^{2}|z|^{2}-uv\overline{z}$ which is impossible. Hence $v$ and $z$ must be zero and with this hypothesis (Cplx) is satisfied.
\begin{rmk}\label{rmk: 5}
    In both cases $u=v=z=0$; $A\neq 0$ and $v=z=A=0$ these conditions (and then (Cplx)) are preserved by any Hermitian curvature flow in the family HCF.
\end{rmk}
Now, setting $v=z=0$, we get the following elements of the curvature tensor:
\begin{align*}
    \Omega_{1\overline{1}1\overline{1}}&=\frac{1}{2}\frac{r^{4}}{t^{2}}(A-1)(\overline{A}-1)\;;\\
    \Omega_{1\overline{1}3\overline{3}}&=-\frac{1}{2}\frac{(A-1)(\overline{A}-1)r^{4}s^{2}-((A-1)\overline{A}-A-3)r^{2}|u|^{2}}{r^{2}s^{2}-|u|^{2}}\;,
\end{align*}
showing that in both cases $u=0$ and $A=0$ the curvature tensor is neither non-negative nor non-positive.

\subsubsection{Solvmanifolds in Families (Sv)}
Recall the complex structure equations for this family:
\begin{equation*}
    d\varphi^{1}=-\varphi^{3\overline{3}},\hspace{0.4em}
    d\varphi^{2}=\frac{\sqrt{-1}}{2}\varphi^{12}+\frac{1}{2}\varphi^{1\overline{3}}-\frac{\sqrt{-1}}{2}\varphi^{2\overline{1}},\hspace{0.4em} 
    d\varphi^{3}=-\frac{\sqrt{-1}}{2}\varphi^{13}+\frac{\sqrt{-1}}{2}\varphi^{3\overline{1}}
\end{equation*}
Consider the terms $\Omega_{233\overline{3}}$ and $\Omega_{123\overline{1}}$: if we set 
\begin{equation*}
    \Omega_{233\overline{3}}=\frac{s^{4}|z|^{2}}{32\sqrt{-1} \det \Xi}=0\;,
\end{equation*} 
we get $z=0$, but then $\Omega_{123\overline{1}}=-\frac{r^{2}s^{2}-|u|^{2}}{4t^{2}}\neq 0$; thus (Cplx) is never satisfied.

\section*{Acknowledgements}
I would like to thank my advisor Daniele Angella for many helpful suggestions and for his constant support and encouragement. I am also grateful to Simone Diverio, Francesco Pediconi and Luis Ugarte Vilumbrales for their comments and suggestions.

\printbibliography

\end{document}